\documentclass[12pt]{amsart}
\usepackage{amsmath, amssymb, amsthm}
\usepackage[margin=1in]{geometry}
\usepackage{tikz}
\usepackage{url}
\usepackage[matrix,arrow,curve,frame]{xy}
\usetikzlibrary{decorations.pathreplacing}
\usepackage[colorlinks=true]{hyperref}

\newtheorem{theorem}{Theorem}[section]

\newtheorem{claim}[theorem]{Claim}
\newtheorem{definition}[theorem]{Definition}
\newtheorem{example}[theorem]{Example}
\newtheorem{observation}[theorem]{Observation}

\newtheorem{remark}[theorem]{Remark}
\newtheorem{question}[theorem]{Question}

\newcommand{\SU}{\mathrm{SU}}
\newcommand{\BSU}{\mathrm{BSU}}
\newcommand{\BBSU}{\mathrm{BBSU}}
\newcommand{\BH}{\mathrm{BH}}
\newcommand{\BK}{\mathrm{BK}}
\newcommand{\BE}{\mathrm{BE}}

\newcommand{\BT}{\mathrm{BT}}
\newcommand{\BW}{\mathrm{BW}}
\newcommand{\W}{\mathrm{W}}
\newcommand{\Wt}{\mathrm{Wt}}
\newcommand{\K}{\mathrm{K}}
\newcommand{\HH}{\mathrm{H}}
\newcommand{\T}{\mathrm{T}}
\newcommand{\E}{\mathrm{E}}

\newcommand{\Fr}{\mathrm{Fr}}
\newcommand{\G}{\mathrm{G}}
\newcommand{\A}{\mathrm{A}}
\newcommand{\B}{\mathrm{B}}
\newcommand{\C}{\mathrm{C}}
\newcommand{\D}{\mathrm{D}}
\newcommand{\R}{\mathrm{R}}
\newcommand{\M}{\mathrm{M}}
\newcommand{\N}{\mathrm{N}}
\newcommand{\BN}{\mathrm{BN}}
\newcommand{\BM}{\mathrm{BM}}
\newcommand{\BG}{\mathrm{BG}}
\newcommand{\lra}{\longrightarrow}
\newcommand{\lla}{\longleftarrow}
\newcommand{\llra}[1]{\stackrel{#1}{\lra}}
\newcommand{\llla}[1]{\stackrel{#1}{\lla}}

\title{Stability phenomena for Kac-Moody Groups}
\author{Nitu Kitchloo}
\address{Department of Mathematics, Johns Hopkins University, Baltimore, USA}
\email{nitu@math.jhu.edu}
\date{\today}

\begin{document}

\begin{abstract}
In this article we show that a canonical procedure of extending generalized Dynkin diagrams gives rise to families of Kac-Moody groups that satisfy homological stability. Furthermore, we identify this stable cohomology ring with the ring of stable Weyl-invariants up to a nilpotent extension and away from a finite set of primes. We also briefly sketch some emergent structure that appears on stabilization. Our results are illustrated for the family $\E_n$ which is of interest in String theory. The techniques used involve homotopy decompositions of classifying spaces of Kac-Moody groups. 
\end{abstract}

\maketitle

\tableofcontents

\section{Introduction}
\noindent
The theory of Kac-Moody Lie algebras and the underlying groups is well established at this point \cite{K, KP,Ku}. The complex points of Kac-Moody groups form a natural extension of the class of semi-simple Lie groups, even though they need not be finite dimensional. Concepts like maximal torus, Weyl groups and root systems extend almost by definition to these groups. As topological groups, one may study Kac-Moody groups through homotopical invariants like the classifying space and its cohomology \cite{Ki,BrK}. 

\medskip
\noindent
Recall that the compact Lie groups in the infinite families $\A_n, \B_n, \C_n$ and $\D_n$ are known to homologically stabilize to very interesting groups that admit emergent homotopical structure like Bott periodicity. It is therefore natural to ask about the structure of such families among Kac-Moody groups. In this article we prove homological stability for canonically defined families of Kac-Moody groups and identify the stable cohomology ring with the stable Weyl invariants, up to nilpotent extensions. We also briefly discuss the emergent structure that appears on stabilization. We illustrate our results in the example of the family $\E_n$ that begins with exceptional Lie groups, and is extended further along Kac-Moody groups. Groups in the $\E_n$ family have been suggested as symmetries of various compactifications of 11-dimensional supergravity \cite{BGH,DN,DHN,J,nCL,W}.

\medskip
\noindent
The results described in this article were motivated by a question asked by Ian Agol in private communication. We thank him for raising this interesting question. 

\section{Generalized Cartan matrices and Kac-Moody groups}\label{sec:gcm}

\noindent
Given a finite index set $I$ with $|I| = n$, a \textbf{generalized Cartan matrix} $\A_{I}$ (GCM) is an integral matrix $\A_{I} = (a_{ij})_{i,j \in I}$ satisfying the following properties (\cite{K}):
\begin{itemize}
\item[(i)] $a_{ii} = 2$
\item[(ii)] $a_{ij} \leq 0$ if $i \neq j$
\item[(iii)] $a_{ij} = 0 \Leftrightarrow a_{ji} = 0$
\end{itemize}
If $J \subseteq I$ is a subset, then by $\A_J \subseteq \A_I$ we will mean the sub generalized Cartan matrix obtained by restricting the indices $i,j$ to the subset $J$. 

\medskip
\begin{example}
In he first example, $|I|=2$ and we assume $a,b > 1$. In the second example, $|I| = n$ with only the nonzero entries shown below
\begin{align*}
\A(a,b) &= \begin{pmatrix} 2 & -a \\ -b & 2 \end{pmatrix} \\[1em]
\A_n &= \begin{pmatrix} 
2 & -1 &  & & & 0 \\
-1 & 2 & -1 & & & \\
 & -1 & 2 & -1 & & \\
& & \ddots & \ddots & \ddots &\\
& & & -1 & 2 & -1 \\
0 & & &  & -1 & 2
\end{pmatrix}
\end{align*}
\end{example}
\noindent
One represents a GCM by a \textbf{generalized Dynkin diagram}, which is a graph with $n$ nodes indexed by $I$, and an edge between distinct nodes $i$ and $j$ if $a_{ij} \neq 0$. We label that edge $(a_{ij}, a_{ji})$ or leave it unlabeled if $a_{ij}a_{ji} = 1$.

\medskip
\begin{example}
\begin{align*}
\A(a,b) &\quad\longleftrightarrow\quad 
\begin{tikzpicture}[baseline=0.15ex]
\draw (0,0) -- (0.8,0);
\node[circle,draw,inner sep=2pt,fill=white] (1) at (0,0) {};
\node[circle,draw,inner sep=2pt,fill=white] (2) at (0.8,0) {};
\node[above] at (0.4,0) {\scriptsize $(a, b)$}  ;
\end{tikzpicture} \\[1em]
\A_n &\quad\longleftrightarrow\quad 
\begin{tikzpicture}[baseline=-0.5ex]
\draw (0,0) -- (3.2,0);
\draw (3.35,0) -- (3.55,0);
\draw (4.4,0) -- (4.8,0);
\foreach \x in {0,1,2,3,4}
  \node[circle,draw,inner sep=2pt,fill=white] at (\x*0.8,0) {};
\node at (3.95,0) {$\cdots$};
\node[circle,draw,inner sep=2pt,fill=white] at (4.8,0) {};
\end{tikzpicture}
\end{align*}
\end{example}

\medskip
\noindent
Given a GCM $\A_{I}$, there is an associated simply-connected topological group $\G(\A_I)$ known as the minimal complex Kac-Moody group \cite{Ku}. The group contains a unitary form $\K(\A_{I})$ such that:

\begin{itemize}
\item[(i)] $\K(\A_{I})$ has rank $n$, i.e., there is a compact maximal torus of rank $n$, call it $\T$.

\item[(ii)] $\K(\A_I)$ is natural in $I$, i.e. an inclusion $J \subseteq I$ functorially induces $\K(\A_J) \subseteq \K(\A_I)$.
\end{itemize}

\medskip
\noindent
If $\A_I$ corresponds to the Dynkin diagram of a compact Lie group, then up to isomorphism, $\K(\A_I)$ recovers this group (these diagrams are called \textbf{finite type}).

\medskip
\begin{example}
$\K(\A_n) = \SU(n+1)$
\end{example}

\medskip
\noindent
The classifying space $\BK(\A_I)$ of the group $\K(\A_I)$ can be described in terms of the classifying spaces of all finite type diagrams in $\A_I$. We describe this result below. 

\medskip
\begin{definition} \label{def:1}
Let $\mathcal{S}(\A_I) = \{ J \subseteq I \mid \A_J \text{ is of finite type} \}$ be the poset known as the poset of spherical subsets for $\A_I$. Given $J \in \mathcal{S}(\A_I)$, we define the following subgroup
$$\HH_{J}(\A_I) = \langle \T, \K(\A_J) \rangle \subseteq \K(\A_I).$$
Hence $\HH_{J}(\A_I)$ is a compact subgroup of $\K(\A_I)$ of maximal rank, extending $\K(\A_J)$ by torus of rank $|I| - |J|$. We define $\W(\A_J)$ to be the Weyl group of $\K(\A_J)$. 
\end{definition}

\medskip
\noindent
The following homotopy decomposition describes the structure of the classifying space $\BK(\A_I)$ and that of $\BW(\A_I)$ in terms of finite type diagrams

\medskip
\begin{theorem}[\cite{Ki},\cite{Ki2},\cite{BrK}]\label{thm:hocolim}
The following canonical maps are homotopy equivalences
$$\underset{J \in \mathcal{S}(\A_I)}{\operatorname{hocolim}} \, \BH_J(\A_I) \longrightarrow \BK(\A_I), \quad \quad \underset{J \in \mathcal{S}(\A_I)}{\operatorname{hocolim}} \, \BW(\A_J) \longrightarrow \BW(\A_I).$$ 
\end{theorem}

\bigskip
\section{Families}\label{sec:families}

\noindent
In order to illustrate the stability properties of families of Kac-Moody groups, let us consider the example of the generalized Cartan matrices $\E_n$ represented by the following sequence of generalized Dynkin diagrams:

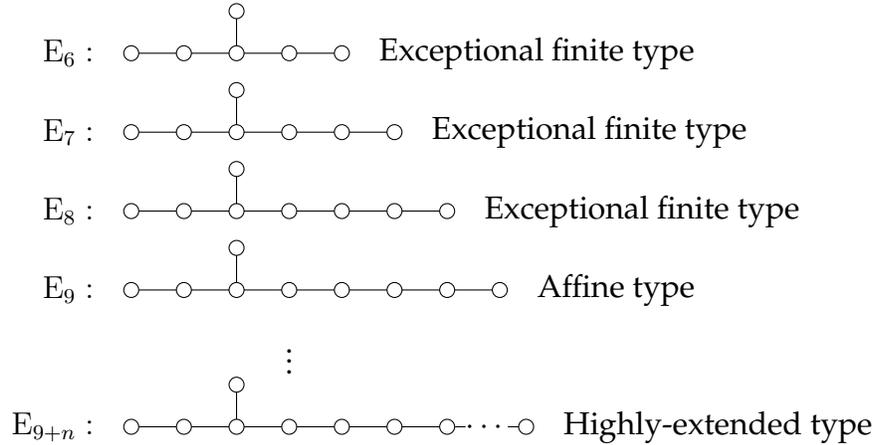
\begin{center}
\begin{tikzpicture}[scale=0.7]
\node[anchor=east] at (-0.5,4.5) {$\E_6:$};
\draw (0,4.5) -- (4,4.5);
\draw (2,4.5) -- (2,5.3);
\foreach \x in {0,1,2,3,4}
  \node[circle,draw,inner sep=2pt,fill=white] at (\x*1,4.5) {};
\node[circle,draw,inner sep=2pt,fill=white] at (2,5.3) {};
\node[anchor=west] at (4.5,4.5) {Exceptional finite type};

\node[anchor=east] at (-0.5,3) {$\E_7:$};
\draw (0,3) -- (5,3);
\draw (2,3) -- (2,3.8);
\foreach \x in {0,1,2,3,4,5}
  \node[circle,draw,inner sep=2pt,fill=white] at (\x*1,3) {};
\node[circle,draw,inner sep=2pt,fill=white] at (2,3.8) {};
\node[anchor=west] at (5.5,3) {Exceptional finite type};

\node[anchor=east] at (-0.5,1.5) {$\E_8:$};
\draw (0,1.5) -- (6,1.5);
\draw (2,1.5) -- (2,2.3);
\foreach \x in {0,1,2,3,4,5,6}
  \node[circle,draw,inner sep=2pt,fill=white] at (\x*1,1.5) {};
\node[circle,draw,inner sep=2pt,fill=white] at (2,2.3) {};
\node[anchor=west] at (6.5,1.5) {Exceptional finite type};

\node[anchor=east] at (-0.5,0) {$\E_9:$};
\draw (0,0) -- (7,0);
\draw (2,0) -- (2,0.8);
\foreach \x in {0,1,2,3,4,5,6,7}
  \node[circle,draw,inner sep=2pt,fill=white] at (\x*1,0) {};
\node[circle,draw,inner sep=2pt,fill=white] at (2,0.8) {};
\node[anchor=west] at (7.5,0) {Affine type};

\node at (3,-1.2) {$\vdots$};

\node[anchor=east] at (-0.5,-2.6) {$\E_{9+n}:$};
\draw (0,-2.6) -- (6,-2.6);
\draw (6.15,-2.6) -- (6.35,-2.6);
\draw (7.15,-2.6) -- (7.5,-2.6);
\draw (2,-2.6) -- (2,-1.8);
\foreach \x in {0,1,2,3,4,5,6}
  \node[circle,draw,inner sep=2pt,fill=white] at (\x*1,-2.6) {};
\node[circle,draw,inner sep=2pt,fill=white] at (2,-1.8) {};
\node at (6.75,-2.6) {$\cdots$};
\node[circle,draw,inner sep=2pt,fill=white] at (7.5,-2.6) {};
\node[anchor=west] at (8,-2.6) {Highly-extended type};
\end{tikzpicture}
\end{center}

\medskip
\begin{definition} \label{def:2}
Using standard terminology, let us denote the Kac-Moody group $\K(\E_n)$ by just $\E_n$. Let $\BE_n$ denote the classifying space of $\E_n$ and define $\BE$ as the topological space
$$\BE := \displaystyle\operatorname*{hocolim}_{n} \,  \BE_n.$$ 
Note: The reader may be tempted to conclude that $\BE$ is equivalent to the classifying space of $\displaystyle\operatorname*{colim}  \E_n$. However, since the groups $\E_n$ are not locally compact for $n > 8$, the colimit of the underlying topological spaces need not have the structure of a topological group. We must therefore define the limiting topological group $\E$ homotopically as $\E := \Omega \BE$. 
\end{definition}

\medskip
\noindent
This leads us to the following natural question: 

\medskip
\begin{question}\label{ques:3.1}
Are the maps of classifying spaces $\BE_n \to \BE_{n+1}$ increasingly connective? In other words, do the spaces $\BE_n$ homologically stabilize to the space $\BE$? 
\end{question}

\noindent
In order to study question \ref{ques:3.1}, let us label the nodes in our  Dynkin diagrams as follows:

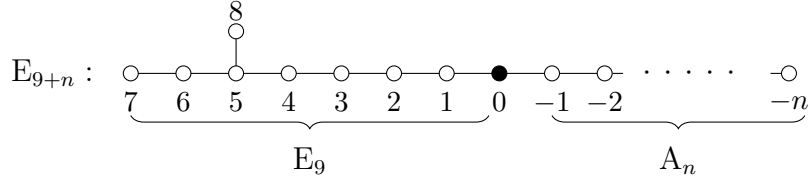
\begin{equation}\label{diagram:E9+n}
\begin{tikzpicture}[scale=0.7]
\node[anchor=east] at (-0.5,0) {$\E_{9+n}:$};

\draw (0,0) -- (9,0);
\draw (9.15,0) -- (9.35,0);
\draw (12.15,0) -- (12.5,0);
\draw (2,0) -- (2,0.8);
\foreach \x in {0,1,2,3,4,5,6,8,9}
  \node[circle,draw,inner sep=2pt,fill=white] at (\x*1,0) {};
\node[circle,draw,inner sep=2pt,fill=black] at (7,0) {};
\node[circle,draw,inner sep=2pt,fill=white] at (2,0.8) {};
\node at (9.8,0) {$\cdot$};
\node at (10.2,0) {$\cdot$};
\node at (10.6,0) {$\cdot$};
\node at (11.0,0) {$\cdot$};
\node at (11.4,0) {$\cdot$};
\node[circle,draw,inner sep=2pt,fill=white] at (12.5,0) {};

\node[below] at (0,-0.15) {\small $7$};
\node[below] at (1,-0.15) {\small $6$};
\node[below] at (2,-0.15) {\small $5$};
\node[below] at (3,-0.15) {\small $4$};
\node[below] at (4,-0.15) {\small $3$};
\node[below] at (5,-0.15) {\small $2$};
\node[below] at (6,-0.15) {\small $1$};
\node[below] at (7,-0.15) {\small $0$};
\node[below] at (8,-0.15) {\small $-1$};
\node[below] at (9,-0.15) {\small $-2$};
\node[below] at (12.5,-0.15) {\small $-n$};
\node[above] at (2,0.8) {\small $8$};

\draw[decorate,decoration={brace,amplitude=5pt,mirror}] (0,-0.8) -- (6.8,-0.8);
\node[below] at (3.4,-1.2) {$\E_9$};

\draw[decorate,decoration={brace,amplitude=5pt,mirror}] (8,-0.8) -- (12.8,-0.8);
\node[below] at (10.4,-1.2) {$\A_n$};
\end{tikzpicture}
\end{equation}

\noindent
The node labeled $0$ is filled in to indicate the fact that it plays a pivotal role as we shall see below. Let us make the following observation about the above diagram: 

\medskip
\begin{observation} \label{obs1}
Let $\mathcal{S}_n$ denote the category of spherical subsets for $\E_{9+n}$. Then we observe that $J$ is an object of $\mathcal{S}_n$ if and only if $J \cap I_0$ is a spherical subset for $\E_9$, where $I_0$ is the subset of indices $i \geq 0$ (so that $\A_{I_0} = \E_9)$. This implies that $\mathcal{S}_n$ contains a {\bf cofinal subcategory} $\mathcal{E}_n$
$$\mathcal{E}_n := \left\{ J \in \mathcal{S}_n \mid J \cap I_0 \text{ is spherical and } i \in J \text{ for all } i < 0 \right\}.$$ By theorem \ref{thm:hocolim}, and the cofinality of $\mathcal{E}_n$, we see that the following map is a homotopy equivalence
$$\underset{J \in \mathcal{E}_n}{\operatorname{hocolim}} \, \BH_J(\E_{9+n}) \longrightarrow \BE_{9+n}.$$
\end{observation}

\medskip
\noindent
Observation \ref{obs1} above allows us to prove the following theorem

\medskip
\begin{theorem} \label{main}
The spaces $\BE_n$ homologically stabilize to $\BE$. In particular, given any degree $k$ and a ring $\R$, the $\R$-modules $\HH^k(\E,\R)$ and $\HH^k(\BE, \R)$ are isomorphic to $\HH^k(\E_n, \R)$ and $\HH^k(\BE_n, \R)$ respectively, for sufficiently large $n$. 
\end{theorem}
\begin{proof}
Recall the homotopy decomposition described in Observation \ref{obs1} above: 
$$\underset{J \in \mathcal{E}_n}{\operatorname{hocolim}} \, \BH_J(\E_{9+n}) \longrightarrow \BE_{9+n}.$$
Notice that the functor sending $J$ to $J \cap I_0$ defines and equivalence between $\mathcal{E}_n$ and  $\mathcal{S}_0$. It follows that $\mathcal{E}_n$ is independent of $n$ upto equivalence. Reindexing the above homotopy colimit over $\mathcal{S}_0$, we may therefore (objectwise) compare the above homotopy colimits as $n$ increases. Now for an fixed object $J_0 \in \mathcal{S}_0$, we obtain an infinite family (indexed by $n$) of {\em compact Lie groups} $\HH_{J_0}(\E_{n+9})$. Since the classifying spaces of all such families are known to homologically stabilize, we conclude from the above homotopy decomposition that the spaces $\BE_{9+n}$ also homologically stabilize. 
\end{proof}

\medskip
\noindent
One can say more about the stable cohomology ring $\HH^\ast(\BE, \R)$. 

\medskip
\begin{theorem}\label{thm2}
Fix a prime $l>5$, and any $\mathbb{Z}_{(l)}$-algebra $\R$.  Then the ring of Weyl invariants $\HH^\ast(\BT_n,\R)^{\W(\E_n)}$ stabilize in $n$, where $\T_n$ denotes the maximal torus of $\E_n$ and $\W(\E_n)$ denotes the Weyl group of $\E_n$ acting on the cohomology of $\BT_n$ in the canonical fashion. Let $\HH^\ast(\BT,\R)^{\W(\E)}$ denote the stable value of Weyl invariants (see Remark \ref{rem2}). 
Then the restriction map
\[ r : \HH^{\ast}(\BE, \R) \longrightarrow \HH^{\ast}(\BT,\R)^{\W(\E)} \]
is a surjection with kernel being the ideal of nilpotent elements in $\HH^\ast(\BE,\R)$. In addition, the ideal of nilpotent elements has exponent less than $9$. 
\end{theorem}
\begin{proof}
For the moment, we let $\R$ denote any arbitrary ring. The homotopy decomposition $\BE_{9+n}$ given in Observation \ref{obs1} gives rise to a cohomologically graded multiplicative spectral  called the Bousfield-Kan spectral sequence:
$$E_2^{i,j} = {\varprojlim_{J \in \mathcal{E}_n}}^{\!i} \, \HH^j(\BH_J(\E_{9+n}), \R) \Rightarrow \HH^{i+j}(\BE_{9+n}, \R).$$ 
The above $E_2$ term is computed via the standard simplicial resolution of $\mathcal{E}_n$ to calculate the derived functors of inverse limit over $\mathcal{E}_n$:
\[ E_1^{i,j} = \bigoplus_{J_1 < \cdots < J_i \in \mathcal{E}_n} \HH^j(\BH_{J_1}(\E_{9+n}), \R). \]

\medskip
\noindent
Since the categories $\mathcal{E}_n$ are all equivalent to $\mathcal{S}_0$, and because the longest sequence of non-trivial composable morphisms in $\mathcal{S}_0$ has length less than $9$, we notice that the terms $E_2^{i,j}$ of the above spectral sequence are trivial if $ i \geq 9$. From the multiplicative structure of the spectral sequence, we notice that the product of any $9$ elements in $\HH^\ast(\BE_{9+n},\R)$ supported on a non-zero column must be trivial. 

\medskip
\noindent
The groups $\HH_J(\E_{9+n})$ are easy to identify for any $J \in \mathcal{E}_n$. By Definition \ref{def:1}, these groups are torus extensions of a semi-simple, simply connected Lie group. This semi-simple factor in $\HH_J(\E_{9+n})$ has a Dynkin diagram which can be expressed as a disjoint union of a subdiagram of $\E_9$ with a diagram in the family $\A_n,\B_n,\C_n$ or $\D_n$. Hence one easily verifies that $\HH^\ast(\BH_J(\E_{9+n}), \mathbb{Z})$ has no $l$-torsion for $l > 5$. For such prime $l$, we deduce that the $E_1$-term of the Bousfield-Kan spectral sequence with coefficients in $\mathbb{Z}_{(l)}$ inject into its rationalization. 

\medskip
\noindent
Now Let $p$ be any prime  so that $\W(\E_{9+n})$ has no $p$-torsion. Then we observe that the unstable Adams operation $\psi^p$ (see Definition \ref{def:4}) acts on the Bousfield-Kan spectral sequence. If we choose coefficients in $\mathbb{Z}_{(l)}$ with $l > 5$, then the action of $\psi^p$  on the $E_1$-term is detected rationally. Hence $\psi^p$ acts by multiplication by $p^j$ on the term $E_2^{i,2j}$. Notice that if $\R$ is any $\mathbb{Z}_{(l)}$-module, the lack of torsion implies that $\HH^\ast(\BH_J(\E_{9+n}), \R) = \HH^\ast(\BH_J(\E_{9+n}), \mathbb{Z}_{(l)}) \otimes \R$. In particular, the action of $\psi^p$ has the same description with coefficients in $\R$.

\medskip
\noindent
Let us now pick $p$ so that it generates the cyclic group $(\mathbb{Z}/l)^\times$. Since $\psi^p$ commutes with the differentials, we have the equalities for $x \in E_{2r-1}^{i,2j}$:
\[ p^{j-r+1} d_{2r-1}(x) = \psi^p d_{2r-1}(x) = d_{2r-1} \psi^p(x) = d_{2r-1} p^j(x) = p^j d_{2r-1}(x). \]
In particular we have:
\[ p^{j-r+1}(p^{r-1}-1) d_{2r-1} (x) = 0. \]
The above equality implies that $(l-1)$ must divide $(r-1)$. But since $l-1$ is larger than $4$, and since the terms of the spectral sequence are trivial beyond the 8th column, we see that there can be no non-trivial differentials. In other words, the spectral sequence collapses at $E_2$. It also follows that the edge homomorphism maps $\HH^\ast(\BE_{9+n},\R)$ surjectively onto the zero-column (which is easily seen to have no nilpotent elements) $${\varprojlim_{J \in \mathcal{E}_n}}^{\!0} \, \HH^j(\BH_J(\E_{9+n}), \R).$$
We conclude that the edge homomorphism is surjective, with kernel precisely the ideal of nilpotent elements (which has exponent less than $9$). Next, we identify the above inverse limit with the Weyl-invariants. This will also prove that the invariants stabilize in $n$. 

\medskip
\noindent
We start by recalling the well-known fact that if $\K$ is a compact connected Lie group, and $l$ is a any odd prime so that $\HH^\ast(\BK,\mathbb{Z})$ has no $l$-torsion, then the following restriction map $$ r : \HH^\ast(\BK,\R) \longrightarrow \HH^\ast(\BT,\R)^{\W(\K)}$$ is an isomorphism where $\R$ is any $\mathbb{Z}_{(l)}$-algebra, and $\T \subseteq \K$ denotes the maximal torus with Weyl group $\W(\K)$. To prove this fact, notice that the lack of torsion shows that $\HH^\ast(\BK,\R)$ is evenly graded. It follows from this that the Serre spectral sequence for the fibration $\K/\T \rightarrow \BT \rightarrow \BK$ collapses. This shows that $\HH^\ast(\BK,\R)$ is a subring of $\HH^\ast(\BT,\R)^{\W(\K)}$. Now let us consider the $\W(\K)$-action on the $E_\infty$-term of the above Serre spectral sequence. Invoking \cite{Ki3}(Corollary 5.8), we see that $\HH^\ast(\K/\T,\R)^{\W(\K)} = \R$ and hence we see that $$E_{\infty}^{\W(\K)} = (\HH^\ast(\K/\T,\R) \otimes \HH^\ast(\BK,\R))^{\W(\K)} = \HH^\ast(\BK,\R).$$ An easy filtration argument now shows that the map $r$ above must also be a surjection. 

\medskip
\noindent
Invoking the above fact, one sees that the following restriction map is an isomorphism $$r  :\HH^j(\BH_J(\E_{9+n}), \R) \longrightarrow \HH^\ast(\BT_{9+n},\R)^{\W(\HH_J(\E_{9+n}))}$$ 
for all $J \in \mathcal{E}_n$. 
Using the above isomorphism we obtain a sequence of equalities
\[  {\varprojlim_{J \in \mathcal{E}_n}}^{\!0} \, \HH^j(\BH_J(\E_{9+n}), \R) = \bigcap_{J \in \mathcal{E}_n} \HH^\ast(\BT_{9+n}, \R)^{\W(\HH_J(\E_{9+n}))} = \bigcap_{J \in \mathcal{S}_n} \HH^\ast(\BT_{9+n}, \R)^{\W(\HH_J(\E_{9+n}))}, \]
where we have identified elements of the inverse limit inside $\HH^\ast(\BT,\R)$, and used the cofinality of $\mathcal{E}_n$ in the second equality. Now notice that the groups $\W(\HH_J(\E_{9+n}))$ are generated by the reflections $r_i$ for singletons $i \in I_{9+n}$. These singletons are elements of $\mathcal{S}_n$. Hence we can identify the above intersection with $\HH^\ast(\BT_{9+n},\R)^{\W(\E_{9+n})}$. 
\end{proof}

\medskip
\noindent
In the discussion above, we have illustrated our results for the family $\E_n$. However, the arguments apply much more generally as we now proceed to show. 

\medskip
\begin{definition}\label{StableM}
Let $\M_n$ be a family of Kac-Moody groups with the property that for some fixed $n_0$, the Dynkin diagram of $\M_{n_0}$ admits a distinguished node labeled $0$ and so that the Dynkin diagram for $\M_{n_0+n}$ is obtained by extending the node $0$ by the Dynkin diagram of $\A_n$ along nodes labeled by the negative integers as illustrated in the diagram (\ref{diagram:E9+n}) for $\E_{9+n}$ shown above. Furthermore, assume that the family $\M_n$ satisfies the property described in Observation \ref{obs1}. In other words, we assume that a subset $J$ is a spherical poset for $\M_{n_0+n}$ if and only if $J \cap I_0$ is a spherical poset for $\M_{n_0}$ where $I_0$ denotes the set of nodes in the Dynkin diagram for $\M_{n_0}$. 
\end{definition}

\medskip
\begin{remark} \label{rem1}
Families of the form $\M_n$ considered above are extremely easy to construct. For instance, we may start with the Dynkin diagram of any Kac-Moody group, and increasingly extend any chosen node linearly in one direction. It is easy to see that this procedure will either define an infinite family of compact Lie groups, or eventually define a diagram $\M_{n_0}$ so that all subsequent extensions $\M_{n_0+n}$ will satisfy the conditions described above. For instance, if one begins with the Dynkin diagram of $\A_4$ and performs the above procedure starting with the second node, then one recovers the family $\E_n$ with $n_0 = 9$ as is evident from the diagram (\ref{diagram:E9+n}) for $\E_{9+n}$ shown above. 
\end{remark}

\medskip
\noindent
The proof of the following theorem is identical to the one given for Theorem \ref{main}. 

\medskip
\begin{theorem}\label{main2}
Given a family $\M_n$ as in Definition \ref{StableM} above, the spaces $\BM_n$ homologically stabilize to $\BM$. In particular, given any degree $k$ and a ring $\R$, the $\R$-modules $\HH^k(\M,\R)$ and $\HH^k(\BM, \R)$ are isomorphic to $\HH^k(\M_n, \R)$ and $\HH^k(\BM_n, \R)$ respectively, for sufficiently large $n$. 
\end{theorem}

\medskip
\begin{remark}\label{rem2}
The argument of theorem \ref{main} applied to the homotopy decomposition of theorem \ref{thm:hocolim} for the classifying space of the Weyl groups $\W(\M_n)$ of any family $\M_n$ as in Definition \ref{StableM} above, shows that the classifying spaces $\BW(\M_n)$ also homologically stabilize to the classifying space $\BW(\M)$ where $\W(\M) := \displaystyle\operatorname*{colim}_{n}  \W(\M_n)$.

\end{remark}

\medskip
\noindent
The analog of Theorem \ref{thm2} also holds for any family $\M_n$ as in Definition \ref{StableM} above. The proof of the following theorem is identical to the one given for Theorem \ref{thm2}. 

\medskip
\begin{theorem}\label{main3}
Let $\M_n$ be a family as in Definition \ref{StableM}. Let us fix any odd prime $l$ so that $2l \geq n_0 + 1$ and $\HH^\ast(\BH_J(\M_{n_0}), \mathbb{Z})$ has no $l$-torsion for all $J \in \mathcal{S}_0$\footnote{Both these conditions hold for any $l > n_0+1$. See Remark \ref{rem5}.}. 
Then for any $\mathbb{Z}_{(l)}$-algebra $\R$, the ring of Weyl invariants $\HH^\ast(\BT_n,\R)^{\W(\M_n)}$ stabilize in $n$, where $\T_n$ denotes the maximal torus of $\M_n$ and $\W(\M_n)$ denotes the Weyl group of $\M_n$ acting on the cohomology of $\BT_n$ in the canonical fashion. Let $\HH^\ast(\BT,\R)^{\W(\M)}$ denote the stable value of Weyl invariants. Then the restriction map $$ r : \HH^\ast(\BM,\R) \longrightarrow \HH^\ast(\BT,\R)^{\W(\M)}$$ is surjective, with the kernel being the ideal of nilpotent elements in $\HH^\ast(\BM,\R)$. In addition, the ideal of nilpotent elements has exponent less than $n_0$. 
\end{theorem}

\bigskip
\section{The emergent structure}\label{sec:structure}

\noindent
In this section, we briefly sketch some of the emmergent structure that arises when we stabilize the classifying spaces of the families considered above. As before, we will illustrate our results in our prototypical example of the family $\E_n$, though the constructions given below apply in full generality. 

\medskip
\noindent
Given the Dynkin diagram of $\E_{n+m}$, for large $n$, let us construct a sub-Dynkin diagram by removing the $m$'th node from the right end of the Dynkin diagram for $\E_{n+m}$. The result is a  disconnected diagram given by $\E_n \coprod \A_{m-1}$. The inclusion of the sub-Dynkin diagram defines the group homomorphisms $\E_n \times \SU(m) \longrightarrow \E_{n+m}$: 

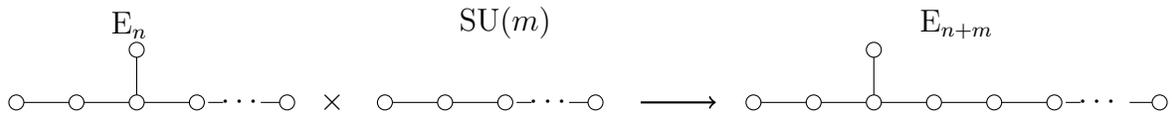
\medskip

\begin{center}
\begin{tikzpicture}[baseline=-0.5ex]
\node[anchor=south] at (1.5,1.5) {$\E_n$};
\draw (0,0.8) -- (2.4,0.8);
\draw (2.55,0.8) -- (2.75,0.8);
\draw (3.25,0.8) -- (3.6,0.8);
\draw (1.6,0.8) -- (1.6,1.5);
\foreach \x in {0,1,2,3}
  \node[circle,draw,inner sep=2pt,fill=white] at (\x*0.8,0.8) {};
\node[circle,draw,inner sep=2pt,fill=white] at (1.6,1.5) {};
\node at (3,0.8) {$\cdots$};
\node[circle,draw,inner sep=2pt,fill=white] at (3.6,0.8) {};

\node at (4.2,0.8) {$\times$};

\node[anchor=south] at (6.5,1.5) {$\SU(m)$};
\draw (4.9,0.8) -- (6.5,0.8);
\draw (6.65,0.8) -- (6.85,0.8);
\draw (7.35,0.8) -- (7.7,0.8);
\foreach \x in {0,1,2}
  \node[circle,draw,inner sep=2pt,fill=white] at (4.9+\x*0.8,0.8) {};
\node at (7.1,0.8) {$\cdots$};
\node[circle,draw,inner sep=2pt,fill=white] at (7.7,0.8) {};

\draw[->,thick] (8.3,0.8) -- (9.3,0.8);

\node[anchor=south] at (12.5,1.5) {$\E_{n+m}$};
\draw (9.8,0.8) -- (13.8,0.8);
\draw (13.95,0.8) -- (14.15,0.8);
\draw (14.8,0.8) -- (15.2,0.8);
\draw (11.4,0.8) -- (11.4,1.5);
\foreach \x in {0,1,2,3,4,5}
  \node[circle,draw,inner sep=2pt,fill=white] at (9.8+\x*0.8,0.8) {};
\node[circle,draw,inner sep=2pt,fill=white] at (11.4,1.5) {};
\node at (14.4,0.8) {$\cdots$};
\node[circle,draw,inner sep=2pt,fill=white] at (15.2,0.8) {};
\end{tikzpicture}
\end{center}
\noindent
Classifying the above gives rise to a map of spaces:
$$\BE_n \times \BSU(m) \longrightarrow \BE_{n+m}.$$
Now consider the following diagram with the horizontal maps being the ones constructed above and the vertical ones induced by the sequential inclusions of $\BE_n$:
\[
\xymatrix{
 \BE_n \times \BSU(m) \ar[d] \ar[r] &  \BE_{n+m} \ar[d] \\
\BE_{n+1} \times \BSU(m)   \ar[r] & \BE_{n+m+1} .
}
\]
This diagram commutes up to canonical homotopy. This homotopy is given by conjugation with the element $\sigma_m \in \SU(m+1) \subset \E_{n+m+1}$ that keeps $\E_{n}$ fixed and shifts the bottom block diagonal copy of $\SU(m) \subset \SU(m+1)$ to the top one. In the standard representation $\mathbb{C}^{m+1}$ of $\SU(m+1)$ the element $\sigma_m$ is given by the following tansformation in terms of the standard basis $\{ e_1,e_2,\ldots,e_{m+1} \}$:
$$ \sigma_m (e_i) = e_{i-1}, \quad i > 1, \quad \sigma_m(e_1) = (-1)^m e_{m+1}.$$
Let us identify $\sigma_{m_1}$ with $\sigma_{m_1} \times 1_{m_2}$, and $\sigma_{m_2}$ with $1_{m_1} \times \sigma_{m_2}$ in terms of the blocks of $\SU(m_1+m_2+1)$. Then it is easy to verify that they satsify $\sigma_{m_2} \sigma_{m_1} = \sigma_{m_1+m_2}$. 

\medskip
\noindent
The above compatible homotopies induced by the elements $\sigma_m$ allow us to construct a well defined $A_{\infty}$-action of the topological monoid $\coprod_m \BSU(m)$ on a $\mathbb{Z}$-graded disjoint union the spaces $\BE$
$$(\mathbb{Z} \times \BE) \times (\coprod_m \times \BSU(m)) \longrightarrow (\mathbb{Z} \times \BE),$$
with the action of $\BSU(m)$ shifting each component of $\BE$ by $m$. 
On group completing \cite{MS}, we obtain a right action of $\mathbb{Z} \times \BSU$ on $\mathbb{Z} \times \BE$:
$$(\mathbb{Z} \times \BE) \times (\mathbb{Z} \times \BSU) \longrightarrow (\mathbb{Z} \times \BE).$$
Restricting the above action to the zero component, and taking homotopy orbits we obtain a principal $\BSU$-fibration: 
$$\BSU \longrightarrow \BE \longrightarrow {\BE}_{h\BSU}$$
which is classfied by the canonical map ${\BE}_{h\BSU} \longrightarrow \BBSU$. Taking loop spaces, we get an extension of topological groups (up to homotopy):
$$\SU \longrightarrow \E \longrightarrow \E/\SU := \Omega({\BE}_{h\BSU}).$$

\medskip
\begin{observation}\label{observation2}
The above emergent structure can be interpreted as saying that $\E$-bundles admit a principal action by stable special unitary bundles and that the homotopy orbits under this action is given by principal $\E/\SU$-bundles. It would be very interesting to have a geometric understanding of this observation. 
\end{observation}

\medskip
\section{Appendix}

\medskip
\noindent
In this section we recall the construction of the unstable Adams operations on $\BK(\A)$ (\cite{Ki4}), which is required  to prove Theorem \ref{thm2}. 

\medskip
\noindent
Given a generalized Cartan matrix $\A_I$, recall the subgroups $\HH_J(\A_I) \subseteq \K(A_I)$ for $J \in \mathcal{S}(\A_I)$ as given in Definition \ref{def:1}. These groups admit complexifications $\G_J(\A_I) \subseteq \G(\A_I)$ called the Levi subgroups of the minimal Kac-Moody group $\G(\A_I)$. Indeed, one has split-forms $\G_J(\A_I)_{\mathbb Z}$ for these subgroups, so that one may take points of $\G_J(\A_I)$ over arbitrary rings. 

\medskip
\noindent
Given an odd prime $p$, let $\overline{\mathbb F}_p$ denote the algebraic closure of the prime field $\mathbb{F}_p$ and let $\Wt(\overline{\mathbb F}_p)$ denote the ring of Witt-vectors over $\overline{\mathbb{F}}_p$. We fix an embedding of rings $\Wt(\overline{\mathbb {F}}_p) \subset \mathbb{C}$ and extend the Frobenius automorphism $\Fr_p$ of $\Wt({\overline{\mathbb F}}_p)$ to an automorphism $\Fr_p$ of $\mathbb C$.

\medskip
\begin{remark} \label{rem3}
Notice that the Teichm\"uller lift of the units $\overline{\mathbb F}_p^{\times} \subset \Wt(\overline{\mathbb {F}}_p)^{\times} \subset \mathbb{C}^{\times}$ lands in the roots of unity in $\mathbb{C}$, and that the automorphism $\Fr_p$ of $\mathbb{C}$ acts via the degree $p$-map on this image. 
\end{remark}

\medskip
\noindent
In \cite{FM} (Theorem 1.4 and Prop.2.3), it is established that there is an equivalence of spaces 
\begin{equation}\label{zigzag0} \prod_{q \neq p} \BG_J(\A_I)(\overline{\mathbb F}_p) \, \hat{}_q \llra{d} \prod_{q \neq p} \BG_J(\A_I) \, \hat{}_q.
\end{equation}
Let $\N_J(\T)(\overline{\mathbb F}_p) \subset \G_J(\A_I)(\overline{\mathbb F}_p)$ denote the $\overline{\mathbb{F}}_p$-points of the normalizer of the (split) torus in $\G_J(\A_I)_{\mathbb Z}$. We observe that there is a canonical map that extends the Teichm\"uller lift 
\[ \iota : \N_J(\T)(\overline{\mathbb F}_p) \subset \N_J(\T), \] 
where $\N_J(\T)$ denotes the complex points of the normalizer of the split torus endowed with the analytic topology. Hence we have a zigzag of maps
\begin{equation}\label{zigzag1} \BN_J(\T) \llra{j} \prod_{q \neq p} \BN_J(\T)\, \hat{}_q \llla{\iota} \prod_{q \neq p} \BN_J(\T)(\overline{\mathbb F}_p) \, \hat{}_q \llra{k} 
\prod_{q \neq p} \BG_J(\A_I)(\overline{\mathbb F}_p) \, \hat{}_q
\end{equation}
where $j$ denotes the completion map, and the leftward moving map is induced by $\iota$ and is seen to be and equivalence since $\iota : \BN_J(\T)(\overline{\mathbb F}_p) \rightarrow \BN_J(\T)$ induces an isomorphism in mod $q$-homology for any $q \neq p$. Using Remark \ref{rem3}, we observe that all maps in zigzag (\ref{zigzag1}) are equivariant with respect to the respective self-maps $\Fr_p$, with $\Fr_p$ defined on $\BN_J(\T)$ as the map induced by the endomorphism of $\N_J(\T)$ that extends the degree $p$ self-map on $\T$ and fixes the Weyl group. Since $p$ is chosen to be an odd prime, such an extension is canonical and agrees with the map $\Fr_p$ defined on $\N_J(\T)(\overline{\mathbb F}_p)$ along the Teichumller lift $\iota$.  

\medskip
\noindent
Let us fix a functorial construction for homotopy pullbacks, and define $\tilde{\BN}_J(\T)$ as the homotopy pullback: 
\begin{equation}\label{zigzag2}
\xymatrix{
 \tilde{\BN}_J(\T) \ar[d]^{\tilde{\iota}} \ar[r]^{\tilde{j} \quad \quad} &  \prod_{q \neq p} \BN_J(\T)(\overline{\mathbb F}_p)\, \hat{}_q \ar[d]^{\iota} \\
\BN_J(\T)   \ar[r]^{j \quad \quad } & \prod_{q \neq p} \BN_J(\T)\, \hat{}_q .
}
\end{equation}
By the functoriality of the construction, we observe that $\tilde{\BN}_J(\T)$ supports a self-map $\Fr_p$ and that all maps are $\Fr_p$-equivariant. Furthermore, $\tilde{\iota}$ is an equivalence since $\iota$ was an equivalence. 

\medskip
\noindent
We now define an $\Fr_p$-equivariant space $\tilde{\BG}_J(\A_I)_{1/p}$ as the homotopy pullback: 
\begin{equation}\label{zigzag3}\xymatrix{
 \tilde{\BG}_J(\A_I)_{1/p} \ar[d]^{\tilde{\eta}} \ar[r] &  \prod_{q \neq p} \BG_J(\A_I)(\overline{\mathbb F}_p)\, \hat{}_q \ar[d] \\
\tilde{\BN}_J(\T)_{\mathbb Q}   \ar[r]^{(d \circ k \circ\tilde{j})_{\mathbb Q} \quad \quad \quad} & (\prod_{q \neq p} \BG_J(\A_I)(\overline{\mathbb F}_p)\, \hat{}_q)_{\mathbb Q} .
}
\end{equation}

\medskip
\noindent
For the rest of this section, we assume that $p$ has the property that the Weyl group $\W(\A_I)$ has no elements of $p$-torsion. 

\medskip
\begin{remark}\label{rem5}
We note that the condition that $\W(\A_I)$ has no elements of $p$-torsion defines a finite set of primes since $\W(\A_I)$ is a subgroup of a general linear group on a lattice of rank $|I|$, which can easily be shown to have no torsion for any $p > |I|+1$.
\end{remark}

\medskip
\noindent
With $p$ as above, we conclude: 

\begin{claim}\label{main4}
The following homotopy pullback is an $\Fr_p$-equivariant model for the space $\BG_J(A)$
\[\xymatrix{
 \tilde{\BG}_J(\A_I) \ar[d] \ar[r] &  \tilde{\BG}_J(\A_I)_{1/p} \ar[d]^{\tilde{\eta}} \\
\tilde{\BN}_J(\T)_{(p)}   \ar[r] & \tilde{\BN}_J(\T)_{\mathbb Q} .
}
\]
Note: The construction of $\Fr_p$ shows that its action on $\BN_J(\T)_{\mathbb Q}$, which is equivalent to $\tilde{\BG}_J(\A_I)_{\mathbb Q}$, is induced by the degree $p$ map on $\T$. 
\end{claim}
\begin{proof}
Given the map $d$ of (\ref{zigzag0}), the map $k$ of (\ref{zigzag1}) and the map $\tilde{j}$ of (\ref{zigzag2}), it is easy to confirm that the composite $(d \circ k \circ\tilde{j})$ is equivalent to the map induced by the inclusion of the maximal torus $\BN_J(\T) \subset \BG_J(\A_I)$ followed by the completion map at all primes $q \neq p$. Furthermore, since $\W(\A)$ contains no $p$-torsion, the same holds for the Weyl group of $\G_J(\A_I)$. It follows that $\tilde{\BN}_J(\T)_{(p)}$ and $\tilde{\BN}_J(\T)_{\mathbb Q}$ are equivalent to the spaces $\BG_J(\A_I)_{(p)}$ and $\BG_J(\A)_{\mathbb Q}$ respectively. It follows from the arithmetic fracture square \cite{BK}(Ch. VI, 8.1) that the homotopy pullback (\ref{zigzag3}) defines a model for $\BG_J(\A_I)$ localized away from the prime $p$. Then invoking the fracture square \cite{BK}(Ch.V, 6.3) shows that our claim above defines a model for $\BG_J(\A_I)$. Since all maps used in our constructions were $\Fr_p$-equivariant and out homotopy pullback constructions were functorial, the above model for $\BG_J(\A_I)$ also admits a self map induced by the maps $\Fr_p$.  
\end{proof}

\medskip
\begin{definition}\label{def:4}
Invoking the evident naturality of our constructions with respect to inclusions $J \subseteq K$ in the category $\mathcal{S}(\A_I)$, we define the following space equivalent to $\BK(\A_I)$:
$$\tilde{\BK}(\A_I) := \underset{J \in \mathcal{S}(\A_I)}{\operatorname{hocolim}} \, \tilde{\BG}_J(\A_I). $$
Naturality also shows that the self-maps $\Fr_p$ constructed in Claim \ref{main4} describe an endomorphism of the functor on $\mathcal{S}(\A_I)$ given by $J \mapsto \tilde{\BG}_J(\A_I)$. Therefore, we obtain a self-map of $\tilde{\BK}(\A_I)$ which we denote by $\psi^p$, and call the {\bf unstable Adams operation}.
\end{definition}

\medskip
\begin{remark}\label{rem4}
The above construction of unstable Adams operations corrects a mistake in the construction of these operations given in \cite{Ki4} (section 3), where we incorrectly used the $\Wt(\overline{\mathbb{F}}_p)$-points of $\G_J(\A_I)_{\mathbb Z}$ instead of the $\overline{\mathbb{F}}_p$-points when constructing a version of diagram \ref{zigzag3}. 
\end{remark}

\end{document}